\documentclass[a4paper,12pt]{article}

\usepackage{amsfonts}
\usepackage{amscd,color}
\usepackage{amsmath,amsfonts,amssymb,amscd}
\numberwithin{equation}{section}
\usepackage{indentfirst,graphicx,epsfig}
\usepackage{graphicx}
\input{epsf}
\usepackage{graphicx}
\usepackage{epstopdf}
\usepackage{caption}
\usepackage{subfigure}
\usepackage{mathrsfs}
\usepackage{ifpdf}
\usepackage{enumerate}
\usepackage{array}
\allowdisplaybreaks
\newtheorem{thm}{Theorem}[section]

\newtheorem{problem}[thm]{Problem}

\newtheorem{lem}[thm]{Lemma}

\newenvironment {proof} {\noindent{\textbf {Proof.}}}{\hfill$\Box$}
\newcommand{\ml}{l\kern-0.55mm\char39\kern-0.3mm}

\title{\textbf{The maximum spectral radius of $\theta_{2,2,3}$-free graphs with given size}}
\author{{\small Jing Gao, Xueliang Li} \\
{\small  Center for Combinatorics and LPMC}\\
{\small Nankai University, Tianjin 300071, China}\\
{\small gjing1270@163.com, lxl@nankai.edu.cn} \\
}
\date{}
\begin{document}
\maketitle
\begin{abstract}
A theta graph $\theta_{r,p,q}$ is the graph obtained by connecting two distinct vertices with three internally disjoint paths of length $r,p,q$, where $q\geq p\geq r\geq1$ and $p\geq2$. A graph is $\theta_{r,p,q}$-free if it does not contain $\theta_{r,p,q}$ as a subgraph. The maximum spectral radius of $\theta_{1,p,q}$-free graphs with given size has been determined for any $q\geq p\geq2$. Zhai, Lin and Shu [Spectral extrema of graphs with fixed size: cycles and complete bipartite graphs, European J. Combin. 95 (2021) 103322] characterized the extremal graph with the maximum spectral radius of $\theta_{2,2,2}$-free graphs having $m$ edges. In this paper, we consider the maximum spectral radius of $\theta_{2,2,3}$-free graphs with size $m$ and characterize the extremal graph.
\\[2mm]
\end{abstract}

\section{\bf Introduction}
Let $G$ be an undirected simple graph with vertex set $V(G)$ and edge set $E(G)$, where $n:=|G|=|V(G)|$ and $m:=e(G)=|E(G)|$ are the order and the size of $G$, respectively. The adjacency matrix of a connected graph $G$ is defined as $A(G)=(a_{u,v})_{n\times n}$ where $a_{u,v}=1$ if $uv\in E(G)$ and $a_{u,v}=0$ otherwise. The spectral radius $\lambda(G)$ of $G$ is the largest eigenvalue of $A(G)$. Given a graph $F$, a graph $G$ is said to be $F$-free if it does not contain $F$ as a subgraph. Let $\mathcal{G}(m, F)$ denote the set of $F$-free graphs with $m$ edges and without isolated vertices. The Brualdi-Hoffman-Tur\'{a}n type problem \cite{Brualdi-Hoffman-1985} is to determine the maximum spectral radius of $F$-free graphs with given size.  This problem has attracted wide attention recently, see \cite{Li-Zhao-2024,Li-Zhai-2024,Nikiforov-2002,Nikiforov-2006,Nikiforov,Yu-Li-2024,Zhai-Lin-2021}.

A theta graph, say $\theta_{r,p,q}$, is the graph obtained by connecting two distinct vertices with three internally disjoint paths of length $r,p,q$, where $q\geq p\geq r\geq1$ and $p\geq2$. About $\theta_{r,p,q}$-free graphs, the Brualdi-Hoffman-Tur\'{a}n type problem has been determined completely for $r=1$. First, Sun et al. \cite{Sun-Li-2023} confirmed the graphs having the largest spectral radius among all $\theta_{1,2,3}$-free and $\theta_{1,2,4}$-free graphs with odd size, respectively. Fang and You \cite{Fang-You-2023} characterized the extremal graph with maximum spectral radius of $\theta_{1,2,3}$-free graphs with even size. Liu and Wang \cite{Liu-Wang-2024} characterized the extremal graph with maximum spectral radius of $\theta_{1,2,4}$-free graphs with even size. Later, Lu et al. \cite{Lu-Lu-2024} characterized the extremal graph with the largest spectral radius of $\theta_{1,2,5}$-free graphs. For $q\geq5$, Li et al. \cite{Li-Zhai-2024} determined the largest spectral radius of $\theta_{1,2,q}$-free graphs. Recently, Gao and Li \cite{Gao-Li-2024} gave the largest spectral radius of $\theta_{1,3,3}$-free graphs. For $q\geq p\geq3$ and $p+q\geq7$, Li et al. \cite{Li-Zhao-2024} obtained the largest spectral radius of $\theta_{1,p,q}$-free graphs. In the same paper, Li et al. \cite{Li-Zhao-2024} proposed a problem about $\theta_{r,p,q}$-free graphs where $q\geq p\geq r\geq2$.
\begin{problem}\cite{Li-Zhao-2024}
How can we characterize the graphs among $\mathcal{G}(m,\theta_{r,p,q})$ having the largest spectral radius for $q\geq p\geq r\geq2$?
\end{problem}
For $r=p=q=2$, we have $\theta_{2,2,2}\cong K_{2,3}$. Zhai et al. \cite{Zhai-Lin-2021} determined the extremal graph for $K_{2,r}$-free graphs with $r\geq3$.
\begin{thm}\cite{Zhai-Lin-2021}
If $G\in\mathcal{G}(m,K_{2,r+1})$ with $r\geq2$ and $m\geq16r^2$, then $\lambda(G)\leq\sqrt{m}$, and equality holds if and only if $G$ is a star.
\end{thm}
In this paper, we give an upper bound of the spectral radius of $\theta_{2,2,3}$-free graphs and characterize the unique graph with the maximum spectral radius among $\mathcal{G}(m, \theta_{2,2,3})$.

\begin{thm}\label{main theorem}
Let $G\in\mathcal{G}(m,\theta_{2,2,3})$ with $m\geq57$. Then $\lambda(G)\leq\frac{1+\sqrt{4m-3}}{2}$ and equality holds if and only if $G\cong K_2\vee\frac{m-1}{2}K_1$.
\end{thm}
\section{Preliminaries}
At the beginning of this section, we give some notations and terminology. Readers are referred to \cite{Bondy-Murty-2008} and \cite{Brouwer-Haemers-2012}. For any vertex $v\in V(G)$, we denote by $N(v)$ or $N_G(v)$ the neighborhood set of $v$ in $G$ and $N[v]=N(v)\cup\{v\}$. Let $d(v)$ or $d_G(v)$ be the degree of a vertex $v$ in $G$. For any two subsets $X,Y\subseteq V(G)$, we denote $N_X(Y)=\cup_{v\in Y}N(v)\cap X$. Let $e(X,Y)$ denote the number of all edges of $G$ with one end vertex in $X$ and the other in $Y$. Particularly, let $e(X):=e(X,X)$. Denote by $G[X]$ the subgraph of $G$ induced by $X$. Given two disjoint graphs $G$ and $H$, denote $G\cup H=(V(G)\cup V(H), E(G)\cup E(H))$. let $G\vee H$ be the graph obtained from $G\cup H$ by joining each vertex of $G$ to each vertex of $H$. As usual, let $P_n$, $C_n$, $K_{1,n-1}$ and $K_n$ be the path, the cycle, the star and the complete graph on $n$ vertices, respectively. Let $K_{1,n-1}+e$ be the graph obtained from $K_{1,n-1}$ by adding one edge within its independent set and $K_n-e$ be a graph obtained from $K_n$ by deleting any one edge.

For a matrix (or vector) $A$, $A > 0 (\geq 0)$ means that all its entries are positive (nonnegative). Here, we state the famous Perron-Frobenius theorem.
\begin{lem}[Perron-Frobenius Theorem]\cite{Brouwer-Haemers-2012}\label{PF}
Let $A\geq0$ be an irreducible symmetric matrix. Then the largest eigenvalue $\lambda(A)$ of $A$ is a real number, and the entries of eigenvector corresponding to $\lambda(A)$ are all positive.
\end{lem}

Note that $A(G)$ is irreducible and nonnegative for a connected graph $G$. By Lemma \ref{PF}, there exists a unique positive unit eigenvector $\mathbf{x}$ corresponding to $\lambda(G)$, which is called Perron vector of $G$. Let $\mathbf{x}$ be the Perron vector of $G$ with coordinate $x_v$ corresponding to the vertex $v\in V(G)$. A vertex $u^\ast$ is said to be an extremal vertex if $x_{u^\ast}=\max\{x_u|u\in V(G)\}$.

A cut vertex of a graph is a vertex whose deletion increases the number of components. A graph is called 2-connected, if it is a connected graph without cut vertices.

\begin{lem}\cite{Zhai-Lin-2021}\label{lem:Zhai-Lin-2021}
Let $G$ be a graph in $\mathcal{G}(m,F)$ with the maximum spectral radius. If $F$ is a 2-connected graph and $u^\ast$ is an extremal vertex of $G$, then $G$ is connected and $d(u)\geq2$ for any $u\in V(G)\setminus N[u^\ast]$.
\end{lem}

The following result is about the largest spectral radius of triangle-free graphs which will be used in the subsequent section.
\begin{lem}\cite{Nikiforov-2002,Nosal-1970}\label{lem:Nikiforov}
Let $G$ be a graph with $m$ edges. If $G$ is triangle-free, then $\lambda(G)\leq\sqrt{m}$. Equality holds if and only if $G$ is a complete bipartite graph.
\end{lem}

\section{Proof of Theorem \ref{main theorem}}
Let $G^\ast$ be the extremal graph with the maximum spectral radius among all graphs in $\mathcal{G}(m,\theta_{2,2,3})$. For convenience, denote $\lambda=\lambda(G^\ast)$. By Lemma \ref{lem:Zhai-Lin-2021}, we know that $G^\ast$ is connected. In the view of Lemma \ref{PF}, there is the Perron vector $\mathbf{x}$. Let $u^\ast$ be the extremal vertex of $G^\ast$. Note that $K_2\vee\frac{m-1}{2}K_1$ is $\theta_{2,2,3}$-free, we have
$$
\lambda\geq\lambda\left(K_2\vee\frac{m-1}{2}K_1\right)=\frac{1+\sqrt{4m-3}}{2}.
$$
Denote $U=N_{G^\ast}(u^\ast)$ and $W=V(G^\ast)\setminus N_{G^\ast}[u^\ast]$. Let $U_0$ be the set of all isolated vertices in the induced subgraph $G^\ast[U]$, and $U_+=U\setminus U_0$ be the set of all vertices with degree at least one in $G^\ast[U]$. Let $W_H=N_W(V(H))$ for any subgraph $H$ of $G^\ast[U]$. Since $\lambda(G^\ast)\mathbf{x}=A(G^\ast)\mathbf{x}$, we have
$$
\lambda x_{u^\ast}=\sum_{u\in U}x_u=\sum_{u\in U_+}x_u+\sum_{u\in U_0}x_u.
$$
Furthermore, we also have $\lambda^2(G^\ast)\mathbf{x}=A^2(G^\ast)\mathbf{x}$. It follows that
\begin{align*}
\lambda^2x_{u^\ast}=|U|x_{u^\ast}+\sum_{u\in U_+}d_U(u)x_u+\sum_{w\in W}d_U(w)x_w.
\end{align*}
Therefore,
\begin{align*}
(\lambda^2-\lambda)x_{u^\ast}&=|U|x_{u^\ast}+\sum_{u\in U_+}(d_U(u)-1)x_u+\sum_{w\in W}d_U(w)x_w-\sum_{u\in U_0}x_u.
\end{align*}
Recall that $\lambda\geq\frac{1+\sqrt{4m-3}}{2}$. It is easy to get that $\lambda^2-\lambda\geq m-1$. Then
\begin{align*}
|U|x_{u^\ast}+\sum_{u\in U_+}(d_U(u)-1)x_u+\sum_{w\in W}d_U(w)x_w-\sum_{u\in U_0}x_u\geq(m-1)x_{u^\ast}.
\end{align*}
Since $m=|U|+e(U_+)+e(U,W)+e(W)$, we have
\begin{align*}
\sum_{u\in U_+}(d_U(u)-1)\frac{x_u}{x_{u^\ast}}+\sum_{w\in W}d_U(w)\frac{x_w}{x_{u^\ast}}\geq e(U_+)+e(U,W)+e(W)+\sum_{u\in U_0}\frac{x_u}{x_{u^\ast}}-1.\tag{1}
\end{align*}
Let $\mathcal{H}$ be the set of all non-trivial components in $G^\ast[U]$. Note that $G^\ast$ is $\theta_{2,2,3}$-free. This implies that $G^\ast[U]$ contains no double star $S_{1,2}$, which is a tree with a central edge $uv$, 1 leaves connected to $u$ and 2 leaves connected to $v$. It follows that every element $H$ in $\mathcal{H}$ is $K_{1,r}$ where $r\geq1$, $K_{1,3}+e$, $K_4-e$, $K_4$, $P_k$ where $k\geq4$ or $C_l$ where $l\geq3$.

\begin{lem}\label{pro1}
Let $H$ be a component of $G^\ast[U]$ which contains a cycle of length at least four. If $W_H\neq\emptyset$, then $d_U(w)\leq2$ for any $w\in W_H$.
\end{lem}

\begin{proof}
Assume that $d_U(w_0)\geq3$ for some $w_0\in W_H$. Let $C_l$ be the cycle of $H$ where $l\geq4$. We have $V(C_l)=V(H)$. Since $w_0\in W_H$, without loss of generality, suppose $w_0\in N_W(u_1)$ where $u_1\in V(C_l)$. Note that $d_U(w_0)\geq3$. Suppose $u_2,u_3\in N_U(w_0)$. Since $l\geq4$, there is at least a vertex $u_i\in\{u_1,u_2,u_3\}$ such that $u_i$ has a neighbor $u_4\in V(C_l)$ different from $\{u_1,u_2,u_3\}$. Hence $G^\ast[u^\ast,u_1,u_2,u_3,u_4,w_0]$ contains a $\theta_{2,2,3}$, which is a contradiction. We complete the proof.
\end{proof}

Let $W_0=\{w\in W|d_W(w)=0\}$. By Lemma \ref{pro1}, if $w\in W_0\cup N_W(C_l)$ where $l\geq4$, then $d(w)\leq2$.

\begin{lem}\label{no at least four cycle}
$G^\ast[U]$ contains no any cycle of length at least four.
\end{lem}

\begin{proof}
Suppose that $G^\ast[U]$ contains $C_l$ where $l\geq4$. Let $\mathcal{H'}$ be the family of components of $G^\ast[U]$ each of which contains cycle of length at least four as a subgraph, then $\mathcal{H}\setminus\mathcal{H'}$ is the family of other components of $G^\ast[U]$ each of which is $K_{1,r}$ where $r\geq1$, $K_{1,3}+e$, $P_k$ where $k\geq4$, $C_3$. Therefore, for each $H\in \mathcal{H}\setminus\mathcal{H'}$, we have $e(H)\leq|H|$. It is clear that
$$
\sum_{u\in V(H)}(d_H(u)-1)x_u\leq(2e(H)-|H|)x_{u^\ast}\leq e(H)x_{u^\ast}.
$$
For any $H\in\mathcal{H'}$, $H$ is $K_4-e$, $K_4$ or $C_l$ where $l\geq4$. In the following, we consider the two cases.

{\bf Case 1. }$W_H=\emptyset$.

If $H\cong C_l=u_1u_2\cdots u_lu_1$ where $l\geq4$, we have \begin{align*}
\begin{cases}
\lambda x_{u_1}= x_{u_l}+x_{u_2}+x_{u^\ast},\\
\lambda x_{u_2}= x_{u_1}+x_{u_3}+x_{u^\ast},\\
\vdots\\
\lambda x_{u_l}= x_{u_{l-1}}+x_{u_1}+x_{u^\ast}.
\end{cases}
\end{align*}
Then $\lambda(x_{u_1}+x_{u_2}+\cdots+x_{u_l})=2(x_{u_1}+x_{u_2}+\cdots+x_{u_l})+lx_{u^\ast}$.  Therefore,
$$
\sum_{u\in V(H)}(d_H(u)-1)x_u=\sum_{i=1}^lx_{u_i}=\frac{l}{\lambda-2}x_{u^\ast}.
$$
Since $m\geq57$, we have $\lambda\geq8$. Hence
$$
\sum_{u\in V(H)}(d_H(u)-1)\frac{x_u}{x_{u^\ast}}<(e(H)-1).
$$

If $H\cong K_4$ or $K_4-e$, suppose $V(H)=\{u_1,u_2,u_3,u_4\}$.
For any vertex $u_i\in\{u_1,u_2,u_3,u_4\}$, we have $d_H(u_i)\leq3$. Therefore, $\lambda x_{u_i}\leq x_{u^\ast}+3x_{u^\ast}=4x_{u^\ast}$. It follows that $x_{u_i}\leq\frac{4}{\lambda}x_{u^\ast}$ for any $i\in\{1,2,3,4\}$. Hence, according to $\lambda\geq8$, we have
\begin{align*}
\sum_{u\in V(H)}(d_H(u)-1)x_u=2\sum_{i=1}^4x_{u_i}\leq\frac{32}{\lambda}x_{u^\ast}\leq4x_{u^\ast}<(e(H)-1)x_{u^\ast}
\end{align*}
for $H\cong K_4$, and
\begin{align*}
\sum_{u\in V(H)}(d_H(u)-1)x_u<2\sum_{i=1}^4x_{u_i}\leq\frac{32}{\lambda}x_{u^\ast}\leq4x_{u^\ast}=(e(H)-1)x_{u^\ast}
\end{align*}
for $H\cong K_4-e$. So $\sum_{u\in V(H)}(d_H(u)-1)\frac{x_u}{x_{u^\ast}}<(e(H)-1)$ when $W_H=\emptyset$.

{\bf Case 2. }$W_H\neq\emptyset$.

If $H\cong C_l=u_1u_2\cdots u_lu_1$ where $l\geq4$, then
\begin{align*}
\begin{cases}
\lambda x_{u_1}= x_{u_l}+x_{u_2}+x_{u^\ast}+\sum_{w\in N_{W}(u_1)}x_w,\\
\lambda x_{u_2}= x_{u_1}+x_{u_3}+x_{u^\ast}+\sum_{w\in N_{W}(u_2)}x_w,\\
\vdots\\
\lambda x_{u_l}= x_{u_{l-1}}+x_{u_1}+x_{u^\ast}+\sum_{w\in N_{W}(u_l)}x_w,
\end{cases}
\end{align*}
we have
\begin{align*}
\lambda(x_{u_1}+x_{u_2}+\cdots+x_{u_l})=lx_{u^\ast}+2(x_{u_1}+x_{u_2}+\cdots+x_{u_l})+\sum_{i=1}^l\sum_{w\in N_{W}(u_i)}x_w.
\end{align*}
That is,
\begin{align*}
x_{u_1}+x_{u_2}+\cdots+x_{u_l}=\frac{l}{\lambda-2}x_{u^\ast}+\frac{1}{\lambda-2}\sum_{u\in V(H)}\sum_{w\in N_W(u)}x_w.
\end{align*}
Thus, by $\lambda\geq8$, we obtain
\begin{align*}
\sum_{u\in V(H)}(d_H(u)-1)\frac{x_u}{x_{u^\ast}}&=\frac{x_{u_1}+x_{u_2}+\cdots+x_{u_l}}{x_{u^\ast}}\\
&=\frac{l}{\lambda-2}+\frac{1}{\lambda-2}\sum_{u\in V(H)}\sum_{w\in N_W(u)}\frac{x_w}{x_{u^\ast}}\\
&<e(H)-1+\frac{1}{\lambda-2}\sum_{u\in V(H)}\sum_{w\in N_W(u)}\frac{x_w}{x_{u^\ast}}.
\end{align*}

If $H\cong K_4-e$, then let $V(H)=\{u_1,u_2,u_3,u_4\}$. Without loss of generality, we suppose that $d_H(u_2)=d_H(u_4)=3$. Then
\begin{align*}
\begin{cases}
\lambda x_{u_1}= x_{u_2}+x_{u_4}+x_{u^\ast}+\sum_{w\in N_{W}(u_1)}x_w,\\
\lambda x_{u_2}= x_{u_1}+x_{u_3}+x_{u_4}+x_{u^\ast}+\sum_{w\in N_{W}(u_2)}x_w,\\
\lambda x_{u_3}= x_{u_2}+x_{u_4}+x_{u^\ast}+\sum_{w\in N_{W}(u_3)}x_w,\\
\lambda x_{u_4}= x_{u_1}+x_{u_2}+x_{u_3}+x_{u^\ast}+\sum_{w\in N_{W}(u_4)}x_w,
\end{cases}
\end{align*}
we have
\begin{align*}
\lambda(x_{u_1}+x_{u_2}+x_{u_3}+x_{u_4})&=4x_{u^\ast}+2(x_{u_1}+x_{u_2}+x_{u_3}+x_{u_4})+x_{u_2}+x_{u_4}+\sum_{i=1}^4\sum_{w\in N_{W}(u_i)}x_w\\
&\leq6x_{u^\ast}+2(x_{u_1}+x_{u_2}+x_{u_3}+x_{u_4})+\sum_{i=1}^4\sum_{w\in N_{W}(u_i)}x_w.
\end{align*}
That is,
\begin{align*}
x_{u_1}+x_{u_2}+x_{u_3}+x_{u_4}\leq\frac{6}{\lambda-2}x_{u^\ast}+\frac{1}{\lambda-2}\sum_{u\in V(H)}\sum_{w\in N_W(u)}x_w.
\end{align*}
Note that $\lambda\geq8$. It follows that
\begin{align*}
\sum_{u\in V(H)}(d_H(u)-1)\frac{x_u}{x_{u^\ast}}&=\frac{x_{u_1}+2x_{u_2}+x_{u_3}+2x_{u_4}}{x_{u^\ast}}\\
&\leq2+\frac{6}{\lambda-2}+\frac{1}{\lambda-2}\sum_{u\in V(H)}\sum_{w\in N_W(u)}\frac{x_w}{x_{u^\ast}}\\
&<e(H)-1+\frac{1}{\lambda-2}\sum_{u\in V(H)}\sum_{w\in N_W(u)}\frac{x_w}{x_{u^\ast}}.
\end{align*}

If $H\cong K_4$ with $V(H)=\{u_1,u_2,u_3,u_4\}$, we have $d_H(u_i)=3$ for any $i\in\{1,2,3,4\}$. Therefore,
\begin{align*}
\begin{cases}
\lambda x_{u_1}\leq 3x_{u^\ast}+x_{u^\ast}+\sum_{w\in N_{W}(u_1)}x_w,\\
\lambda x_{u_2}\leq 3x_{u^\ast}+x_{u^\ast}+\sum_{w\in N_{W}(u_2)}x_w,\\
\lambda x_{u_3}\leq 3x_{u^\ast}+x_{u^\ast}+\sum_{w\in N_{W}(u_3)}x_w,\\
\lambda x_{u_4}\leq 3x_{u^\ast}+x_{u^\ast}+\sum_{w\in N_{W}(u_4)}x_w,
\end{cases}
\end{align*}
we have
\begin{align*}
\lambda(x_{u_1}+x_{u_2}+x_{u_3}+x_{u_4})\leq 16x_{u^\ast}+\sum_{i=1}^4\sum_{w\in N_{W}(u_i)}x_w.
\end{align*}
Thus,
\begin{align*}
\sum_{u\in V(H)}(d_H(u)-1)\frac{x_u}{x_{u^\ast}}&=\frac{2(x_{u_1}+x_{u_2}+x_{u_3}+x_{u_4})}{x_{u^\ast}}\\
&\leq\frac{32}{\lambda}+\frac{2}{\lambda}\sum_{u\in V(H)}\sum_{w\in N_W(u)}\frac{x_w}{x_{u^\ast}}.
\end{align*}
Since $\lambda\geq 8$, we obtain
\begin{align*}
\sum_{u\in V(H)}(d_H(u)-1)\frac{x_u}{x_{u^\ast}}<e(H)-1+\frac{2}{\lambda}\sum_{u\in V(H)}\sum_{w\in N_W(u)}\frac{x_w}{x_{u^\ast}}.
\end{align*}

It is easy to get $\frac{2}{\lambda}\geq\frac{1}{\lambda-2}$ for $\lambda\geq8$. Then $\sum_{u\in V(H)}(d_H(u)-1)\frac{x_u}{x_{u^\ast}}<e(H)-1+\frac{2}{\lambda}\sum_{u\in V(H)}\sum_{w\in N_W(u)}\frac{x_w}{x_{u^\ast}}$ for each $H\in \mathcal{H'}$. Thus,
\begin{align*}
\sum_{u\in U_+}(d_U(u)-1)\frac{x_u}{x_{u^\ast}}&=\sum_{H\in\mathcal{H}\setminus\mathcal{H'}}\left(\sum_{u\in V(H)}(d_H(u)-1)\frac{x_u}{x_{u^\ast}}\right)+\sum_{H\in\mathcal{H'}}\left(\sum_{u\in V(H)}(d_H(u)-1)\frac{x_u}{x_{u^\ast}}\right)\\
&<\sum_{H\in\mathcal{H}\setminus\mathcal{H'}}e(H)+\sum_{H\in\mathcal{H'}}(e(H)-1)+\sum_{H\in\mathcal{H'}}\frac{2}{\lambda}\sum_{u\in V(H)}\sum_{w\in N_W(u)}\frac{x_w}{x_{u^\ast}}\\
&\leq e(U_+)-|\mathcal{H'}|+\frac{2}{\lambda}\sum_{w\in \cup_{H\in\mathcal{H'}}W_H}d_U(w)\frac{x_w}{x_{u^\ast}}.
\end{align*}
Let $W_1=\cup_{H\in\mathcal{H'}}W_H\cap W_0$. By Lemma \ref{pro1}, we have $d_U(w)\leq2$ for any $w\in W_H$ where $H\in\mathcal{H'}$ and $d(w)\leq2$ for any $w\in W_1$. Therefore, for $w\in W_1$, we get $\lambda x_w\leq 2x_{u^\ast}$. That is, $x_w\leq\frac{2}{\lambda}x_{u^\ast}$ for any $w\in W_1$.
This implies that
\begin{align*}
&\sum_{u\in U_+}(d_U(u)-1)\frac{x_u}{x_{u^\ast}}+\sum_{w\in W}d_U(w)\frac{x_w}{x_{u^\ast}}\\
&< e(U_+)-|\mathcal{H'}|+\frac{2}{\lambda}\cdot2\sum_{w\in\bigcup_{H\in\mathcal{H'}}W_H}\frac{x_w}{x_{u^\ast}}+2\sum_{w\in W_1}\frac{x_w}{x_{u^\ast}}+\sum_{w\in W\setminus W_1}d_U(w)\frac{x_w}{x_{u^\ast}}\\
&=e(U_+)-|\mathcal{H'}|+\frac{4}{\lambda}\sum_{w\in W_1}\frac{x_w}{x_{u^\ast}}+\frac{4}{\lambda}\sum_{w\in \cup_{H\in\mathcal{H'}}W_H\setminus W_1}\frac{x_w}{x_{u^\ast}}+2\sum_{w\in W_1}\frac{x_w}{x_{u^\ast}}+\sum_{w\in W\setminus W_1}d_U(w)\frac{x_w}{x_{u^\ast}}\\
&\leq e(U_+)-|\mathcal{H'}|+\left(\frac{4}{\lambda}+2\right)\cdot\frac{2}{\lambda}e(U,W_1)+\frac{4}{\lambda}\sum_{w\in \cup_{H\in\mathcal{H'}}W_H\setminus W_1}d_W(w)\frac{x_w}{x_{u^\ast}}+e(U,W\setminus W_1)\\
&\leq e(U_+)-|\mathcal{H'}|+\frac{8+4\lambda}{\lambda^2}e(U,W_1)+\frac{8}{\lambda}e(W)+e(U,W\setminus W_1).
\end{align*}
Since $\lambda\geq8$, we have $\frac{8+4\lambda}{\lambda^2}\leq1$ and $\frac{8}{\lambda}\leq1$. Then
\begin{align*}
\sum_{u\in U_+}(d_U(u)-1)\frac{x_u}{x_{u^\ast}}+\sum_{w\in W}d_U(w)\frac{x_w}{x_{u^\ast}}
&< e(U_+)-1+e(U,W_1)+e(W)+e(U,W\setminus W_1)\\
&= e(U_+)-1+e(U,W)+e(W),
\end{align*}
which contradicts with (1). This completes the proof.
\end{proof}

By Lemma \ref{no at least four cycle}, we obtain that every non-trivial component of $G^\ast[U]$ is $K_{1,r}$ where $r\geq1$, $K_{1,3}+e$, $P_k$ where $k\geq4$ or $C_3$.

\begin{lem}\label{no edge in W}
$e(W)=0$.
\end{lem}

\begin{proof}
If $W=\emptyset$, then $e(W)=0$, as desired. So we consider $W\neq\emptyset$ in the following. Suppose to the contrary that $e(W)\geq1$. Since every non-trivial component of $G^\ast[U]$ is a tree or a unicyclic graph, we have $e(U_+)\leq|U_+|$. By inequality (1), we get
\begin{align*}
e(W)&\leq\sum_{u\in U_+}(d_U(u)-1)\frac{x_u}{x_{u^\ast}}+\sum_{w\in W}d_U(w)\frac{x_w}{x_{u^\ast}}-e(U_+)-e(U,W)-\sum_{u\in U_0}\frac{x_u}{x_{u^\ast}}+1\\
&\leq2e(U_+)-|U_+|+e(U,W)-e(U_+)-e(U,W)+1\\
&\leq 1.
\end{align*}
So $e(W)=1$ and $x_w=x_{u^\ast}$ for any $w\in W$ satisfying $d_U(w)\geq1$. Let $E(W)=\{w_1w_2\}$. By Lemma \ref{lem:Zhai-Lin-2021}, we know that $d(w_1)\geq2$ and $d(w_2)\geq2$. This implies that $d_U(w_1)\geq1$ and $d_U(w_2)\geq1$. It follows that $x_{w_1}=x_{w_2}=x_{u^\ast}$. Since $G^\ast$ is $\theta_{2,2,3}$-free, we have $d_U(w_1)+d_U(w_2)\leq4$. Otherwise, there is a vertex $w_i\in\{w_1,w_2\}$ such that $d_U(w_i)\geq3$. Without loss of generality, suppose $w_i=w_1$. Let $u_1,u_2,u_3\in N_U(w_1)$. Note that $d_U(w_2)\geq1$. Let $u_4\in N_U(w_2)$. Then there are at least two vertices of $\{u_1,u_2,u_3\}$ which are different from $u_4$. Assume that $u_1$ and $u_2$ are different from $u_4$. We can find that $G^\ast[u^\ast,u_1,u_2,u_4,w_1,w_2]$ contains a $\theta_{2,2,3}$, a contradiction. Therefore,
\begin{align*}
2\lambda x_{u^\ast}&=\lambda x_{w_1}+\lambda x_{w_2}\\
&=x_{w_2}+\sum_{u\in N_U(w_1)}x_u+x_{w_1}+\sum_{u\in N_U(w_2)}x_u\\
&\leq x_{u^\ast}+4x_{u^\ast}+x_{u^\ast}\\
&=6x_{u^\ast}.
\end{align*}
It yields that $\lambda\leq3$, a contradiction. This completes the proof.
\end{proof}

From now on, if $W_H=\emptyset$, we denote $\sum_{w\in N_{W}(u)}x_w=0$ for any $u\in V(H)$.
\begin{lem}\label{no K3+}
For any $H\in\mathcal{H}$, we have $H\ncong K_{1,3}+e$.
\end{lem}

\begin{proof}
Suppose that there is a component $H\in\mathcal{H}$ such that $H\cong K_{1,3}+e$. Let $V(H)=\{u_1,u_2,u_3,u_4\}$ with $d_H(u_1)=3$ and $d_H(u_4)=1$. 
We first prove that $d_U(w)\leq2$ for any vertex $w\in W_H$ if $W_H\neq\emptyset$. Otherwise, there is a vertex $w_0\in W_H$ such that $d_U(w_0)\geq3$. Since $w_0\in W_H$, we can see that $w_0\in N_W(u_i)$ for some $i\in\{1,2,3,4\}$. Note that $d_U(w_0)\geq3$. Suppose $v_1,v_2\in N_U(w_0)\setminus\{u_i\}$. If $i=1$, then at least one vertex $u_j\in\{u_2,u_3,u_4\}$ is different from $v_1$ and $v_2$. Therefore, $u^\ast v_1w_0, u^\ast v_2w_0, u^\ast u_ju_iw_0$ are three internally disjoint paths of length 2,2,3 between $u^\ast$ and $w_0$, a contradiction. So $w_0\notin N_W(u_1)$. Then $u^\ast v_1w_0, u^\ast v_2w_0, u^\ast u_1u_iw_0$ are three internally disjoint paths of length 2,2,3 between $u^\ast$ and $w_0$, a contradiction. Thus, $d_U(w)\leq2$ for any vertex $w\in W_H$. By Lemmas \ref{lem:Zhai-Lin-2021} and \ref{no edge in W}, we have $d(w)=2$ for any vertex $w\in W_H$. Therefore, $\lambda x_w\leq 2x_{u^\ast}$. That is, $x_w\leq\frac{2}{\lambda}x_{u^\ast}$ for any vertex $w\in W_H$. According to $\lambda \mathbf{x}=A(G^\ast)\mathbf{x}$, we obtain
\begin{align*}
\begin{cases}
\lambda x_{u_1}= x_{u_2}+x_{u_3}+x_{u_4}+x_{u^\ast}+\sum_{w\in N_{W}(u_1)}x_w,\\
\lambda x_{u_2}= x_{u_1}+x_{u_3}+x_{u^\ast}+\sum_{w\in N_{W}(u_2)}x_w,\\
\lambda x_{u_3}= x_{u_1}+x_{u_2}+x_{u^\ast}+\sum_{w\in N_{W}(u_3)}x_w.
\end{cases}
\end{align*}
Thus,
\begin{align*}
(\lambda-2)(x_{u_1}+x_{u_2}+x_{u_3})=3x_{u^\ast}+x_{u_4}+\sum_{i=1}^3\sum_{w\in N_{W}(u_i)}x_w.
\end{align*}
By $\lambda\geq8$, we have
\begin{align*}
x_{u_1}+x_{u_2}+x_{u_3}&\leq\frac{4x_{u^\ast}}{\lambda-2}+\frac{1}{\lambda-2}\sum_{i=1}^3\sum_{w\in N_{W}(u_i)}x_w\\
&< (e(H)-2)x_{u^\ast}+\frac{1}{\lambda-2}\sum_{i=1}^3\sum_{w\in N_{W}(u_i)}x_w.
\end{align*}
Recall that $x_w\leq\frac{2}{\lambda}x_{u^\ast}$ for any vertex $w\in W_H$. We conclude that
\begin{align*}
\sum_{u\in V(H)}(d_H(u)-1)\frac{x_u}{x_{u^\ast}}&=\frac{2x_{u_1}+x_{u_2}+x_{u_3}}{x_{u^\ast}}\\
&<1+e(H)-2+\frac{1}{\lambda-2}\cdot\frac{2}{\lambda}e(H,W_H)\\
&= e(H)-1+\frac{2}{\lambda(\lambda-2)}e(H,W_H).
\end{align*}
Hence,
\begin{align*}
&\sum_{u\in U_+}(d_U(u)-1)\frac{x_u}{x_{u^\ast}}+\sum_{w\in W}d_U(w)\frac{x_w}{x_{u^\ast}}\\
&=\sum_{u\in U_+\setminus V(H)}(d_U(u)-1)\frac{x_u}{x_{u^\ast}}+\sum_{u\in V(H)}(d_U(u)-1)\frac{x_u}{x_{u^\ast}}+\sum_{w\in W_H}d_U(w)\frac{x_w}{x_{u^\ast}}+\sum_{w\in W\setminus W_H}d_U(w)\frac{x_w}{x_{u^\ast}}\\
&< e(U_+\setminus V(H))+e(H)-1+\frac{2}{\lambda(\lambda-2)}e(H,W_H)+\frac{2}{\lambda}e(U,W_H)+e(U,W\setminus W_H)\\
&\leq e(U_+)-1+\frac{2\lambda-2}{\lambda(\lambda-2)}e(U,W_H)+e(U,W\setminus W_H)\\
&\leq e(U_+)-1+e(U,W).
\end{align*}
This is a contradiction. We complete the proof.
\end{proof}

\begin{lem}\label{no C3}
For any $H\in\mathcal{H}$, we have $H\ncong C_3$.
\end{lem}

\begin{proof}
Suppose to the contrary that there is a component $H\in\mathcal{H}$ such that $H\cong C_3$. Let $V(H)=\{u_1,u_2,u_3\}$. If $W_H\neq\emptyset$, we have $d_U(w)\leq3$ for any $w\in W_H$. Otherwise, there is a vertex $w_0\in W_H$ satisfying $d_U(w_0)\geq4$. Without loss of generality, assume $w_0\in N_W(u_1)$. Note that $|V(H)|=3$. Suppose that $v\in N_U(w_0)$ is different from $u_1,u_2,u_3$. Since $d_U(w_0)\geq4$, there is another vertex $v'\in N_U(w_0)\setminus\{u_1,v\}$. Then at least one of $u_2$ and $u_3$ is different from $v'$. Suppose that $u_2$ is different from $v'$. It is easy to find that $u^\ast vw_0, u^\ast v'w_0, u^\ast u_2u_1w_0$ are three internally disjoint paths of length 2,2,3, a contradiction. By Lemma \ref{no edge in W}, we obtain $\lambda x_w\leq 3x_{u^\ast}$ for any vertex $w\in W_H$.
Since
\begin{align*}
\begin{cases}
\lambda x_{u_1}= x_{u_2}+x_{u_3}+x_{u^\ast}+\sum_{w\in N_{W}(u_1)}x_w,\\
\lambda x_{u_2}= x_{u_1}+x_{u_3}+x_{u^\ast}+\sum_{w\in N_{W}(u_2)}x_w,\\
\lambda x_{u_3}= x_{u_1}+x_{u_2}+x_{u^\ast}+\sum_{w\in N_{W}(u_3)}x_w,
\end{cases}
\end{align*}
we get
\begin{align*}
(\lambda-2)(x_{u_1}+x_{u_2}+x_{u_3})=3x_{u^\ast}+\sum_{i=1}^3\sum_{w\in N_{W}(u_i)}x_w.
\end{align*}
Recall that $\lambda\geq8$. It follows that
\begin{align*}
x_{u_1}+x_{u_2}+x_{u_3}&=\frac{3x_{u^\ast}}{\lambda-2}+\frac{1}{\lambda-2}\sum_{i=1}^3\sum_{w\in N_{W}(u_i)}x_w\\
&< (e(H)-1)x_{u^\ast}+\frac{1}{\lambda-2}\sum_{i=1}^3\sum_{w\in N_{W}(u_i)}x_w.
\end{align*}
Since $x_w\leq\frac{3}{\lambda}x_{u^\ast}$ for any vertex $w\in W_H$, we obtain
\begin{align*}
\sum_{u\in V(H)}(d_H(u)-1)\frac{x_u}{x_{u^\ast}}&=\frac{x_{u_1}+x_{u_2}+x_{u_3}}{x_{u^\ast}}\\
&< e(H)-1+\frac{1}{\lambda-2}\cdot\frac{3}{\lambda}e(H,W_H)\\
&=e(H)-1+\frac{3}{\lambda(\lambda-2)}e(H,W_H).
\end{align*}
Hence,
\begin{align*}
&\sum_{u\in U_+}(d_U(u)-1)\frac{x_u}{x_{u^\ast}}+\sum_{w\in W}d_U(w)\frac{x_w}{x_{u^\ast}}\\
&=\sum_{u\in U_+\setminus V(H)}(d_U(u)-1)\frac{x_u}{x_{u^\ast}}+\sum_{u\in V(H)}(d_U(u)-1)\frac{x_u}{x_{u^\ast}}+\sum_{w\in W_H}d_U(w)\frac{x_w}{x_{u^\ast}}+\sum_{w\in W\setminus W_H}d_U(w)\frac{x_w}{x_{u^\ast}}\\
&< e(U_+\setminus V(H))+e(H)-1+\frac{3}{\lambda(\lambda-2)}e(H,W_H)+\frac{3}{\lambda}e(U,W_H)+e(U,W\setminus W_H)\\
&\leq e(U_+)-1+\frac{3\lambda-3}{\lambda(\lambda-2)}e(U,W_H)+e(U,W\setminus W_H)\\
&\leq e(U_+)-1+e(U,W).
\end{align*}
This is a contradiction. We complete the proof.
\end{proof}

\begin{lem}\label{no Pk}
For any $H\in\mathcal{H}$, we have $H\ncong P_k$ where $k\geq4$.
\end{lem}

\begin{proof}
Let $P_k=u_1u_2\cdots u_k$ where $k\geq4$. If $W_H\neq\emptyset$, we show that $d_U(w)\leq2$ for any $w\in W_H$. Suppose that there exists a vertex $w_0\in W_H$ satisfying $d_U(w_0)\geq3$. Assume $v_1,v_2,v_3\in N_U(w_0)$. Since $k\geq4$, there is a vertex $u_i\in V(P_k)$ such that $u_{i-1}\in N_U(w_0)$ or $u_{i+1}\in N_U(w_0)$ and $u_i\notin\{v_1,v_2,v_3\}$. Without loss of generality, suppose $u_{i-1}\in N_U(w_0)$. Then at least two vertices of $v_1,v_2,v_3$ are different from $u_{i-1}$. Suppose the two vertices are $v_1,v_2$. It follows that $u^\ast v_1w_0, u^\ast v_2w_0, u^\ast u_iu_{i-1}w_0$ are three internally disjoint paths of length 2,2,3, a contradiction. Therefore, $d_U(w)\leq2$ for any $w\in W_H$. By Lemma \ref{no edge in W}, we have $d(w)\leq2$ for any $w\in W_H$. This implies that $\lambda x_w\leq2x_{u^\ast}$ for any $w\in W_H$. By
\begin{align*}
\begin{cases}
\lambda x_{u_2}= x_{u_1}+x_{u_3}+x_{u^\ast}+\sum_{w\in N_{W}(u_2)}x_w,\\
\lambda x_{u_3}= x_{u_2}+x_{u_4}+x_{u^\ast}+\sum_{w\in N_{W}(u_3)}x_w,\\
\vdots\\
\lambda x_{u_{k-1}}= x_{u_{k-2}}+x_{u_k}+x_{u^\ast}+\sum_{w\in N_{W}(u_{k-1})}x_w,
\end{cases}
\end{align*}
we have
\begin{align*}
&\lambda(x_{u_2}+x_{u_3}+\cdots+x_{u_{k-1}})\\
&=(k-2)x_{u^\ast}+x_{u_1}+x_{u_2}+2(x_{u_3}+\cdots+x_{u_{k-2}})+x_{u_{k-1}}+x_{u_k}+\sum_{i=2}^{k-1}\sum_{w\in N_{W}(u_i)}x_w\\
&\leq2(k-2)x_{u^\ast}+x_{u_2}+x_{u_3}+\cdots+x_{u_{k-1}}+\sum_{i=2}^{k-1}\sum_{w\in N_{W}(u_i)}x_w.
\end{align*}
Note that $\lambda x_w\leq2x_{u^\ast}$ for any $w\in W_H$. We obtain
\begin{align*}
x_{u_2}+x_{u_3}+\cdots+x_{u_{k-1}}&\leq\frac{2(k-2)}{\lambda-1}x_{u^\ast}+\frac{1}{\lambda-1}\cdot\frac{2}{\lambda}e(H,W_H)x_{u^\ast}\\
&=\frac{2(k-2)}{\lambda-1}x_{u^\ast}+\frac{2}{\lambda(\lambda-1)}e(H,W_H)x_{u^\ast}.
\end{align*}
By $\lambda\geq8$, we have
\begin{align*}
\sum_{u\in V(H)}(d_H(u)-1)\frac{x_u}{x_{u^\ast}}&=\frac{x_{u_2}+x_{u_3}+\cdots+x_{u_{k-1}}}{x_{u^\ast}}\\
&\leq\frac{2(k-2)}{\lambda-1}+\frac{2}{\lambda(\lambda-1)}e(H,W_H)\\
&<e(H)-1+\frac{2}{\lambda(\lambda-1)}e(H,W_H).
\end{align*}
Therefore,
\begin{align*}
&\sum_{u\in U_+}(d_U(u)-1)\frac{x_u}{x_{u^\ast}}+\sum_{w\in W}d_U(w)\frac{x_w}{x_{u^\ast}}\\
&=\sum_{u\in U_+\setminus V(H)}(d_U(u)-1)\frac{x_u}{x_{u^\ast}}+\sum_{u\in V(H)}(d_U(u)-1)\frac{x_u}{x_{u^\ast}}+\sum_{w\in W_H}d_U(w)\frac{x_w}{x_{u^\ast}}+\sum_{w\in W\setminus W_H}d_U(w)\frac{x_w}{x_{u^\ast}}\\
&< e(U_+\setminus V(H))+e(H)-1+\frac{2}{\lambda(\lambda-1)}e(H,W_H)+\frac{2}{\lambda}e(U,W_H)+e(U,W\setminus W_H)\\
&\leq e(U_+)-1+\frac{2}{\lambda-1}e(U,W_H)+e(U,W\setminus W_H)\\
&\leq e(U_+)-1+e(U,W),
\end{align*}
a contradiction. This completes the proof.
\end{proof}

According to Lemmas \ref{no at least four cycle}, \ref{no K3+}, \ref{no C3}, \ref{no Pk}, we get that every element $H$ in $\mathcal{H}$ is $K_{1,r}$ where $r\geq1$. Then $e(H)=|H|-1$. This implies that
\begin{align*}
\sum_{u\in U_+}(d_U(u)-1)\frac{x_u}{x_{u^\ast}}
&\leq \sum_{H\in\mathcal{H}}(2e(H)-|H|)\\
&= \sum_{H\in\mathcal{H}}(e(H)-1)\\
&= e(U_+)-|\mathcal{H}|.
\end{align*}
By (1) and $\sum_{w\in W}d_U(w)\frac{x_w}{x_{u^\ast}}\leq e(U,W)$, we have
\begin{align*}
\sum_{u\in U_+}(d_U(u)-1)\frac{x_u}{x_{u^\ast}}\geq e(U_+)+\sum_{u\in U_0}\frac{x_u}{x_{u^\ast}}-1.
\end{align*}
Combining with the two inequalities, we have
$|\mathcal{H}|+\sum_{u\in U_0}\frac{x_u}{x_{u^\ast}}\leq1$.
Next we finish the proof of Theorem \ref{main theorem}.

{\bf Proof of Theorem \ref{main theorem}.} If $|\mathcal{H}|=0$, then $G^\ast$ is bipartite. By Lemma \ref{lem:Nikiforov}, we have $\lambda\leq\sqrt{m}<\frac{1+\sqrt{4m-3}}{2}$. This contradicts with $\lambda\geq\frac{1+\sqrt{4m-3}}{2}$. So $|\mathcal{H}|=1$. It follows that $U_0=\emptyset$ due to $x_u>0$ for any $u\in V(G^\ast)$ and $\sum_{u\in U_+}(d_U(u)-1)\frac{x_u}{x_{u^\ast}}= e(U_+)-1$. That is, $G^\ast[U]\cong K_{1,r}$. If $r=1$, then $G^\ast[U]$ contains an edge $u_0u_1$. We have $\lambda x_{u^\ast}=x_{u_0}+x_{u_1}$ and $\lambda x_{u_0}=x_{u^\ast}+x_{u_1}+\sum_{w\in N_w(u_0)}x_w$. It yields that $\sum_{w\in N_w(u_0)}x_w=(\lambda+1)(x_{u_0}-x_{u^\ast})\leq0$. Since $\sum_{w\in N_w(u_0)}x_w\geq0$, we obtain $\sum_{w\in N_w(u_0)}x_w=0$. That is, $N_w(u_0)=\emptyset$. By Lemmas \ref{lem:Zhai-Lin-2021} and \ref{no edge in W}, we have $W=\emptyset$. Then $m=3$, a contradiction. So $r\geq2$. Let $U=\{u_0,u_1,\ldots,u_r\}$ with the center $u_0$. Because $d_U(u_0)\geq2$, we have $x_{u_0}=x_{u^\ast}$.
Since
\begin{align*}
\lambda x_{u^\ast}=x_{u_0}+x_{u_1}+\cdots+x_{u_r}
\end{align*}
and
\begin{align*}
\lambda x_{u_0}=x_{u^\ast}+x_{u_1}+\cdots+x_{u_r}+\sum_{w\in N_w(u_0)}x_w,
\end{align*}
we get $\sum_{w\in N_w(u_0)}x_w=0$. Thus, $N_w(u_0)=\emptyset$. If $r=2$, then we have $N(w)=\{u_1,u_2\}$ for any $w\in W$ by Lemma \ref{lem:Zhai-Lin-2021}. It follows that $|W|\leq1$. Otherwise there is a $\theta_{2,2,3}$, contradiction. Therefore, $m=e(G^\ast)=7$, a contradiction. This implies that $r\geq3$. If $W\neq\emptyset$, we have $d(w)\leq2$ for any $w\in W$. Otherwise, suppose that $u_1,u_2,u_3$ are three neighbors of $w_0\in W$. Then $u^\ast u_1w_0, u^\ast u_2w_0, u^\ast u_0u_3w_0$ are three internally disjoint paths of length 2,2,3, a contradiction. Therefore, $\lambda x_w\leq2x_{u^\ast}$ for any $w\in W$. It follows that
\begin{align*}
&\sum_{u\in U_+}(d_U(u)-1)\frac{x_u}{x_{u^\ast}}+\sum_{w\in W}d_U(w)\frac{x_w}{x_{u^\ast}}\\
&\leq(e(U_+)-1)+\frac{2}{\lambda}e(U,W)\\
&< e(U_+)-1+e(U,W),
\end{align*}
a contradiction. Hence, $W=\emptyset$. This implies that $G^\ast\cong K_1\vee K_{1,r}$ with $2r+1=m$. That is, $G^\ast\cong K_2\vee \frac{m-1}{2}K_1$. We complete the proof.\hfill$\Box$
\section*{\bf Acknowledgments}

This work was supported by National Natural Science Foundation of China (Nos.12131013 and 12161141006).\\

\end{document}